\documentclass[11pt]{amsart}

\usepackage{amssymb}

\theoremstyle{plain}
\newtheorem{thm}{Theorem}[section]
\newtheorem{lem}[thm]{Lemma}
\newtheorem{prop}[thm]{Proposition}
\newtheorem{cor}[thm]{Corollary}
\theoremstyle{remark}
\newtheorem{rem}[thm]{Remark}

\newtheorem{dft}[thm]{Definition}
\long\def\symbolfootnote[#1]#2{\begingroup%
\def\thefootnote{\fnsymbol{footnote}}\footnote[#1]{#2}\endgroup}

\begin{document}
\title[Fixed point sets and Lefschetz modules]{On fixed point sets and Lefschetz modules for sporadic simple groups}

\maketitle

\begin{center}
John Maginnis\symbolfootnote[1]{Corresponding author.} and Silvia Onofrei \symbolfootnote[0]{{\it Emails}: maginnis@math.ksu.edu (J. Maginnis), onofrei@math.ksu.edu (S. Onofrei).}

{\it Mathematics Department, Kansas State University,\\
137 Cardwell Hall, Manhattan, Kansas 66506}
\end{center}

\begin{abstract}We consider $2$-local geometries and other subgroup complexes for sporadic simple groups. For six groups, the fixed point set of a noncentral involution is shown to be equivariantly homotopy equivalent to a standard geometry for the component of the centralizer. For odd primes, fixed point sets are computed for sporadic groups having an extraspecial Sylow $p$-subgroup of order $p^3$, acting on the complex of those $p$-radical subgroups containing a $p$-central element in their centers. Vertices for summands of the associated reduced Lefschetz modules are described.\\
\end{abstract}

\section{Introduction}
In this paper we study various subgroup complexes of the sporadic simple groups, including $2$-local geometries, the distinguished Bouc complex and the complex of $p$-centric and $p$-radical subgroups. We find information on the structure of the associated reduced Lefschetz modules, such as vertices and their distribution into the blocks of the group ring. Robinson \cite{rob88} reformulated a theorem of Burry-Carlson \cite{bc82} and Puig \cite{puig81} to relate the indecomposable summands with given vertex of the reduced Lefschetz module to the indecomposable summands with the same vertex of a corresponding Lefschetz module for the subcomplex fixed by the action of the vertex group. Therefore the nature of the fixed point sets is essential for the understanding of the properties of the reduced Lefschetz modules. Our approach will combine this theorem with standard results from modular representation theory.

The best known example of a Lefschetz module is the Steinberg module of a Lie group $G$ in defining characteristic $p$. This module is irreducible and projective and it is the homology module of the associated building. For finite groups in general, the Brown complex of inclusion chains of nontrivial $p$-subgroups (as well as any complex which is $G$-homotopy equivalent to it) has projective reduced Lefschetz module; see \cite[Corollary 4.3]{qu78}. For a Lie group in defining characteristic, the Brown complex is $G$-homotopy equivalent to the building \cite[Theorem 3.1]{qu78}, and thus its Lefschetz module is equal to the corresponding Steinberg module.

Ryba, Smith and Yoshiara \cite{rsy} proved projectivity for Lefschetz modules of $18$ sporadic geometries. The characters of these projective modules (where available) and the decomposition into projective covers of irreducibles are also given. Smith and Yoshiara \cite{sy} approached the study of projective Lefschetz modules in a more systematic way. The projectivity results are obtained via homotopy equivalence with the Quillen complex of nontrivial elementary abelian $p$-subgroups. Their results show that the $2$-local geometries for the sporadic groups of local characteristic $2$ have projective Lefschetz modules.

The collection of nontrivial $p$-centric and $p$-radical subgroups (see Section $3$ for precise definitions) is relevant to both modular representation theory as well as mod-$p$ cohomology of the underlying group. Sawabe \cite[Section 4]{sa05} showed that its reduced Lefschetz module is projective relative to the collection of $p$-subgroups which are $p$-radical but not $p$-centric.

\smallskip
In Section $2$, we assume $p=2$ and discuss six sporadic simple groups:  $M_{12}, J_2, HS, Ru, Suz$ and $Co_3$.
Each has two classes of involutions; the $2$-central involutions are closed under commuting products.
The fixed point sets of the $2$-central involutions are contractible; we determine the structure of the fixed point sets of the noncentral involutions and we relate these fixed point sets to sporadic geometries and buildings associated to a component of the centralizer of the noncentral involution. We also determine the vertices and the distribution of the nonprojective indecomposable summands of the Lefschetz module into the $2$-blocks of the group ring. These results are contained in Theorems $2.1$ and $2.11$.

We will refer to a $p$-group as being distinguished if it contains a $p$-central element in its center; see the paper \cite{mgo3} for further information on distinguished collections. For odd primes, we analyze in detail the distinguished $p$-radical complexes for those sporadic simple groups which have a Sylow $p$-subgroup of order $p^3$. The main results are given in Theorem $3.1$. The homotopy type of the fixed point sets of the subgroups of order $p$ are determined, and information on the structure of the corresponding Lefschetz modules is given. Our computations provide several examples of reduced Lefschetz modules which are acyclic and contain nonprojective indecomposable summands in the principal block of the group ring; these examples are obtained for $p=3$ ($He$ and $M_{24}$), $p=7$ and $Fi'_{24}$ and for $p=13$ and $M$, the Monster group.

In the second part of Section $3$ we determine the properties of the reduced Lefschetz modules associated to the complex of $p$-centric and $p$-radical subgroups. The groups discussed are those sporadic groups with extraspecial Sylow $p$-subgroup and which do not have parabolic characteristic $p$; the groups $J_2$, $M_{24}$ and $He$, with $p=3$ fall into this category.

\smallskip
We use the standard notation for finite groups as in the Atlas \cite{Atlas}. If $p$ is a prime, then $p^n$ denotes an elementary abelian $p$-group of rank $n$ and $p$ a cyclic group of order $p$. For $p$ odd, $p^{1+2}_+$ stands for the extraspecial group of order $p^3$ and exponent $p$. Also $H.K$ denotes an extension of the group $H$ by a group $K$ and $H:K$ denotes a split extension. Regarding the conjugacy classes of elements of order $p$ we follow the Atlas \cite{Atlas} notation. The simplified notation of the form $5A^2$ stands for an elementary abelian group $5^2$ whose $24$ nontrivial elements are all of type $5A$. The notation of the form $7AB$ stands for a group of order $7$ which contains elements from both classes $7A$ and $7B$. A great deal of information on the $p$-local structure of the sporadic simple groups used in Section $3$ can be found in \cite[Section 5.3]{gls3}. In order to determine if certain classes of $p$-elements are closed under taking commuting products we compute the class multiplication coefficient denoted $\xi(x,y,z)$ using the built-in function in GAP \cite{gap}.

For both characters and Brauer characters, we use the notation from the Modular Atlas Homepage \cite{moc}, where $\phi_i$ denotes the Brauer character of a simple module in characteristic $p$ and $\chi_i$ is the character of a simple module in characteristic zero, also given in \cite{Atlas}. Then $P_G(\phi_i)$ denotes the projective cover of $\phi_i$.

\section{Sporadic groups which satisfy the closure property for $p=2$}

\subsection{Fixed point sets}
Let $Co_3$ be the third Conway group, one of the $26$ sporadic simple groups. There is a $2$-local geometry $\Delta$ of $Co_3$, described by Ronan and Stroth \cite{rstr84}, with three types of objects having isotropy groups equal to three of the maximal $2$-local subgroups, $2^.Sp_6(2), 2^{2+6}.3.(S_3 \times S_3)$, and ${2^4}^. A_8$. The authors have shown in \cite[Proposition 6.1]{mgo1} that the fixed point set
$\Delta^z$ of a $2$-central involution $z$ is contractible. However, the fixed point set $\Delta^t$ of a noncentral involution $t$ is not contractible (this is related to the fact that the mod-$2$ reduced Lefschetz module of $\Delta$ is not projective). The
centralizer of a noncentral involution is $C_{Co_3}(t) = 2 \times M_{12}$, so that the Mathieu group $M_{12}$ acts on the fixed point set $\Delta^t$. Also $M_{12}$ has its own $2$-local geometry, with two types of objects, stabilized by $2_+^{1+4}.S_3$ and $4^2:(2 \times S_3$); see for example \cite[Section 8.27]{bs04}.

In the thesis of P. Grizzard \cite{gri}, the character of the Lefschetz module is computed for $Co_3$ and seven other sporadic groups. This yields the Euler characteristics of the fixed point sets. Grizzard noticed that the fixed point set $\Delta^t$ and the $2$-local
geometry for $M_{12}$ have the same Euler characteristic, and conjectured that they are homotopy equivalent. With equal Euler characteristics as evidence, Grizzard conjectures that for six of the eight sporadic groups he investigated, the fixed point set for a noncentral involution is homotopy equivalent to a standard $2$-local geometry for the component of the centralizer of the noncentral involution. The purpose of this section is to prove this conjecture.

\begin{thm} Let $G$ be one of the following six sporadic simple groups: the Mathieu group $M_{12}$, the Hall-Janko-Wales group $J_2$, the Higman-Sims group $HS$, the third Conway group $Co_3$, the Rudvalis group $Ru$, or the Suzuki group $Suz$. Let $H \leq G$ be the component of the centralizer $C_G(t)$ of a noncentral involution $t \in G$. Then $H$ and $G$ have standard $2$-local geometries such that the geometry for $H$ is equivariantly homotopy equivalent to the fixed point set for the involution $t$ acting on the geometry for $G$.
\end{thm}

A $2$-{\it local geometry} for $G$ is a simplicial complex whose vertex stabilizers are maximal $2$-local subgroups of $G$. If $G$ is one of the six sporadic groups mentioned in Theorem $2.1$, the term {\it standard} indicates the fact that we are working with the $2$-local geometries considered in Benson and Smith \cite{bs04}. If $G$ is a Lie group in defining characteristic, then the standard geometry is the corresponding Tits building.

The six groups in Theorem $2.1$ have several important similarities. There are two conjugacy classes of involutions; one class consists of the $2$-{\it central} involutions, those lying in the center of a Sylow $2$-subgroup. The collection of $2$-central involutions is {\it closed}, in the sense that the product of two distinct commuting $2$-central involutions is always an involution of central type. The centralizer of a noncentral involution has a single component $H$, which in these cases is a simple group, normal in the centralizer. Table $2.1$ describes the centralizer $C_G(t)$ of a noncentral involution $t\in G$, and its component $H$.

The {\it Quillen complex} of $G$ is the complex of proper inclusion chains of nontrivial elementary abelian $p$-subgroups of $G$. The {\it Benson complex} is a subcomplex of the Quillen complex; its vertices are elementary abelian $p$-subgroups of $G$ whose elements of order $p$ lie in certain conjugacy classes, generated under commuting products by the $p$-{\it central elements} (those elements of order $p$ which lie in the center of a Sylow $p$-subgroup of $G$). This complex was introduced by Benson \cite{ben94} in order to study the mod-$2$ cohomology of $Co_3$.

\begin{center}
\begin{tabular}{|c|c|c|}
  \hline
  $G$ & $C_G(t)$ & $H$\\
  \hline
  $M_{12}$ & $2 \times S_5$ & $A_5$\\
  $J_2$ & $2^2 \times A_5$ & $A_5$ \\
  $HS$ & $2 \times Aut(A_6)$ & $A_6$ \\
  $Co_3$ & $2 \times M_{12}$ & $M_{12}$ \\
  $Ru$ & $2^2 \times Sz(8)$ & $Sz(8)$ \\
  $Suz$ & $(2^2 \times L_3(4)).2$& $L_3(4)$ \\
  \hline
\end{tabular}

\vspace*{.3cm}
{\it Table 2.1: Centralizers of noncentral involutions and their components}
\end{center}

\begin{lem}[Benson and Smith \cite{bs04}] For the six groups in Theorem 2.1,
the standard 2-local geometry for $G$ is equivariantly homotopy equivalent to the Benson complex.
\end{lem}
\begin{proof} This result is proved case by case in the book of Benson and Smith \cite[Chapter 8]{bs04}.
\end{proof}

For the remainder of this section, let $\Delta$ denote the Benson complex for $G$, where $G$
is one of the six groups in Theorem $2.1$. Since the collection of $2$-central involutions is closed, the simplices of $\Delta$ are chains of elementary abelian $2$-subgroups of purely central type (every involution is $2$-central).

\begin{lem} Let $t$ be a noncentral involution in $G$, and let $H$ be the component of the centralizer $C_G(t)$. Let $\overline{H} = H$ except in the case $G = HS$, where $\overline{H} = S_6 \leq Aut(A_6)$. Then the $2$-central involutions of $G$ which lie in $C_G(t)$ actually lie in $\overline{H}$. These are precisely the $2$-central involutions of $\overline{H}$.
\end{lem}

\begin{proof} For $G = HS$, this result is Lemma $1.7$ of Aschbacher \cite[p. 24]{a03}. For $G=Suz$, see the
paragraph on p. $456$ before Lemma $1$ of Yoshiara \cite{y01}. For the other cases, this result follows from
information in the Atlas \cite{Atlas}; for example, the square of an element of order $4$ is $2$-central. Observe that $A_5$, $Sz(8)$ and $L_3(4)$ have only one conjugacy class of involutions, and in
$S_6$ every involution is $2$-central.
\end{proof}

\begin{lem} The standard $2$-local geometry for $H$ is equivariantly homotopy equivalent to the Benson complex for $\overline{H}$.
\end{lem}

\begin{proof} For $M_{12}$, this was stated in Lemma $2.2$. The Benson complex equals the Quillen complex
for $A_5\simeq L_2(4), S_6 \simeq Sp_4(2), Sz(8)$ and $L_3(4)$. For a group of Lie type in defining characteristic, the Quillen complex is homotopy equivalent to the Tits building; see \cite[Theorem 3.1]{qu78}.
\end{proof}

\begin{dft} Let $\Delta_0^t$ denote the subcomplex of the Benson complex $\Delta$ for $G$ which consists of chains of purely $2$-central elementary abelian $2$-subgroups $A$ satisfying $t \in C_G(A)$ or equivalently $A \leq C_G(t)$.
\end{dft}

By Lemma $2.3$, $\Delta_0^t$ equals the Benson complex for $\overline{H}$, and so $\Delta_0^t$ is equivariantly homotopy equivalent to the standard $2$-local geometry for $H$. Note that $\Delta_0^t$ is a subcomplex of the fixed point set $\Delta^t$, which consists of those chains of elementary abelian $2$-subgroups $A$ of purely central type satisfying $t \in N_G(A)$.

\begin{lem} The subcomplex $\Delta_0^t$ is equivariantly homotopy equivalent to $\Delta^t$.
\end{lem}

\begin{proof} Define a map $f: \Delta ^t \rightarrow \Delta ^t$ by $f(A)=A \cap C_G(t)$, where $A$ is a purely $2$-central elementary abelian $2$-subgroup of $G$ with $t \in N_G(A)$. Observe that the action of $t$ on $A$ must fix at least one nonidentity element, so that $f(A) \not = \lbrace e \rbrace$. Thus $f(A) \in \Delta$, which is closed under passing to nontrivial
subgroups. Clearly $f(A) \in \Delta _0 ^t \subseteq \Delta ^t$. The map $f$ is a $C_G(t)$-equivariant poset map
satisfying $f(A) \leq A$ and therefore $\Delta ^t$ is $C_G(t)$-homotopy equivalent to $\Delta _0^t$, the image of $f$; see \cite[2.2(3)]{gs}.
\end{proof}

We end the section with a summary of the arguments used to prove the theorem:

\begin{proof}[Proof of Theorem 2.1] Let $G$ denote one of the six groups in Theorem $2.1$. Each of these groups has a standard $2$-local geometry which is equivariantly homotopy equivalent to the Benson complex $\Delta$ of $G$ (Lemma $2.2$). By Lemma $2.6$, the fixed point set $\Delta^t$ of a noncentral involution $t$ in $G$ is equivariantly homotopy equivalent with $\Delta_0^t$, which is the Benson complex of the component $H$ of $C_G(t)$, by Lemma $2.3$ and Definition $2.5$. Finally, according to Lemma $2.4$, the Benson complex of $H$ is equivariantly homotopy equivalent to the standard $2$-local geometry for $H$. This concludes the proof of the theorem.
\end{proof}

\subsection{Reduced Lefschetz modules}
Information on the fixed point sets leads to more details on the {\it reduced
Lefschetz module}:
$$\widetilde{L}_G(\Delta; \mathbb{F}_p)= \sum _{i=-1}^{{\rm dim} \Delta} (-1)^i C_i (\Delta ; \mathbb{F}_p)$$
of the augmented chain complex of $\Delta$; here $\mathbb{F}_p$ denotes a field of characteristic $p$, which is a splitting field for $G$ and all its subgroups. The reduced Lefschetz module can also be written in terms of induced modules:
$$\widetilde{L}_G(\Delta: \mathbb{F}_p)= {\sum _{\sigma \in \Delta} (-1)^{ {\rm dim} \sigma} Ind_{G_{\sigma}}^G ( \mathbb{F}_p  }) - \mathbb{F}_p.$$

A theorem due to Burry and Carlson \cite[Theorem 5]{bc82} and to Puig \cite{puig81} was applied by Robinson \cite[in the proof of Corollary 3.2]{rob88} to Lefschetz modules to obtain the following result; also see \cite[Lemma 1]{sa06}:

\begin{thm}[Robinson \cite{rob88}] The number of indecomposable summands of $\widetilde{L}_G(\Delta;\mathbb{F}_p)$ with vertex $Q$ is equal to the number of indecomposable summands of $\widetilde{L}_{N_G(Q)}(\Delta ^Q;\mathbb{F}_p)$ with the same vertex $Q$.
\end{thm}

The proof of this result uses the Green correspondence, and the relationship to the Brauer correspondence permits a conclusion regarding the blocks in which the summands lie.

In what follows we will use this theorem in order to determine the vertices of the reduced Lefschetz module $\widetilde{L}_G(\Delta;\mathbb{F}_2)$ for the Benson complex of one of the six sporadic simple groups discussed in Section $2.1$. We start with a more general result regarding the nature of the fixed point sets under the action of $p$-central elements:

\begin{prop} Let $\Delta$ denote the Benson complex for a group $G$, and let $P$ be any group in the Benson collection. Then the fixed point set $\Delta^P$ is contractible. Further, the vertices of the indecomposable summands of $\widetilde{L}_G(\Delta; \mathbb{F}_p)$ do not contain any $p$-central elements.
\end{prop}

\begin{proof} A {\it collection} $\mathcal{C}$ is a $G$-poset which is closed under the conjugation action of $G$. The order complex $\Delta(\mathcal{C})$ of $\mathcal{C}$ is the simplicial complex with vertex set $\mathcal{C}$ and simplices proper inclusion chains in $\mathcal{C}$. In what follows, $\mathcal{C}$ will denote the {\it Benson collection}; this is the set of $p$-subgroups of $G$ which correspond to vertices in the Benson complex.

Let $\mathcal{C}_0^P= \lbrace Q \in \mathcal{C} | P \leq C_G(Q) \rbrace$, a subposet of $\mathcal{C}^P = \lbrace Q \in \mathcal{C} | P \leq N_G(Q) \rbrace$. Define a poset map
$F: \mathcal{C}^P \rightarrow \mathcal{C}^P$ by $F(Q)=Q \cap C_G(P)$, a nontrivial group since it equals the set of elements of the $p$-group $Q$ fixed under the action of the $p$-group $P$. Since $F(Q) \leq Q$, we have homotopy equivalence between $\Delta(\mathcal{C}^P)$ and the corresponding image under $F$, which equals $\Delta(\mathcal{C}_0^P)$. If $P \leq C_G(Q)$, then $PQ$ is also a group in the Benson collection. Then $\Delta(\mathcal{C}_0^P)$ is conically contractible via $Q \leq PQ \geq P$.  These two stages of the proof can be combined in the string, representing equivariant poset maps: $ Q \geq C_Q(P) \leq P \cdot C_Q(P) \geq P.$

Let $z$ be a $p$-central element in $G$, then $\langle z \rangle$ belongs to the Benson collection. The contractibility of $\Delta^z$ implies that $\Delta ^Q$ is mod-$p$ acyclic for any $p$-group $Q$ containing $z$ (by Smith theory), and thus the reduced Lefschetz module $\widetilde{L}_{N_G(Q)}(\Delta ^Q;\mathbb{F}_p)=0$. Apply Theorem $2.7$ to conclude that the vertices of the indecomposable summands of $\widetilde{L}_G(\Delta; \mathbb{F}_p)$ do not contain any $p$-central elements.
\end{proof}

\begin{cor} With $G$ one of the sporadic simple groups $M_{12}, J_2, HS, Co_3, Ru, Suz$, the vertices of the indecomposable summands of $\widetilde{L}_G(\Delta;\mathbb{F}_2)$ do not contain any $2$-central involutions.
\end{cor}

\begin{rem} For $G = Co_3$, the result was proved in \cite[Theorem 2]{mgo1} using a different approach.
\end{rem}

In what follows, we need only to consider those $2$-groups $Q$ which are purely noncentral. For five of the six groups discussed in this section ($M_{12}, J_2, HS, Ru$ and $Co_3$) the square of any element of
order $4$ is a $2$-central involution. Therefore the purely noncentral $2$-groups are the purely noncentral elementary abelian $2$-groups. These are classified in \cite[Theorems 1.3 and 1.4]{amm}, \cite[Lemmas 3.3 and 3.4]{finr73}, \cite[Table 2]{magl71}, \cite[Section 2.5]{wi83}, \cite[Lemma 5.10]{fin73}.

Now consider the Suzuki group $G=Suz$. There are two classes of involutions, denoted $2A$ (2-central) and $2B$ (noncentral). The centralizer of a noncentral involution $t_0=2B$ is $C_G(t_0)=(2^2 \times L_3(4)):2$. There are four classes of elements of order $4$, and the class $4D$ has its square equal to a noncentral involution $2B$. The maximal purely noncentral elementary abelian $2$-subgroups have order four, and there are $3$ classes of such subgroups, denoted $V_1, V_2$ and $V_3$ in \cite{wi83}. Their centralizers are $C_G(V_1)=2^2 \times L_3(4)$, $C_G(V_2)=2^2 \times 2^{2+4}$, and $C_G(V_3)=2^2 \times 3^2:Q_8$. Denoting $V_1=\langle t_0, t_1 \rangle$, we have $V_2 = \langle t_0, t_1 y \rangle$ with $y$ an involution in $L_3(4)$, and $V_3= \langle t_0, z \rangle$ for $z$ an involution in $C_G(t_0) \setminus C_G(V_1)$. Acting as an outer automorphism on $L_3(4)$, $z$ centralizes $PSU_3(2)=3^2:Q_8$. The element $t_1z$ is of type $4D$, and $D_8=\langle t_0, t_1,z \rangle$ is the dihedral group of order $8$.\\
Let $Q \leq G$ be a purely noncentral $2$-group, and denote $E =\Omega _1 (Z(Q))$. Thus $E$ is a purely noncentral elementary abelian $2$-group and $Q \leq C_G(E)$. If $E \simeq 2^2$ has order four, then $Q =E$ since every involution of $L_3(4)$ is $2$-central in $G$. Assume $E=2=\langle t_0 \rangle$; then $Q \cap C_G(V_1)$ equals either $E, V_1$, or $V_2$. Either $Q \leq C_G(V_1)$, or $Q$ is isomorphic to some extension of the form $2.2$ or $2^2.2$. In the first case, $Q$ is either $4=\langle 4D \rangle$ or $Q=2^2=V_3$. In the second case $Q$ is dihedral since there are no purely noncentral $2^3$, and $\Omega _1 (Z (2 \times 4))=2^2$. Up to conjugacy, the dihedral group contains $V_3 = \langle t_0, z \rangle$, and either $Q = D_8= \langle t_0, t_1, z \rangle$ or $Q = D_8^* = \langle t_0, t_1y,z \rangle$.

The chart in Table $2.2$ lists the conjugacy classes of purely noncentral $2$-subgroups in each of the six sporadic groups (second column) together with their normalizers (third column). The fixed point sets (given in the fourth column) are computed using $\Delta ^Q = (\Delta ^t)^Q$ for $t$ an involution in $Z(Q)$. In most cases $\Delta ^Q$ is homotopy equivalent to a building and the reduced Lefschetz module is the associated Steinberg module, or an ``extended" Steinberg module; see \cite{schmid92}. The case of $2 \times M_{12} \leq Co_3$
will require more discussion, and we need to know the projective summands of the corresponding reduced Lefschetz module for $M_{12}$.

\begin{thm}(a) Let $G$ be either $M_{12}$, $J_2$, $HS$, or $Ru$. Then the reduced Lefschetz module $\widetilde{L}_G(\Delta; \mathbb{F}_2)$ has precisely one nonprojective summand, and that summand has vertex $2^2$ and lies in a block with defect group $2^2$.\\
(b) Let $G = Suz$. Then $\widetilde{L}_G(\Delta; \mathbb{F}_2)$ has precisely one nonprojective summand, and that summand has vertex $D_8$ and lies in a block with defect group $D_8$.\\
(c) Let $G=Co_3$. Then $\widetilde{L}_G(\Delta; \mathbb{F}_2)$ has either two or three nonprojective summands, all lying in a block with defect group $2^3$. One summand has vertex $2^3$. There is either one or two summands with vertex $2$.
\end{thm}

\begin{proof} The results for part (a) and part (b) follow from Theorem $2.7$ and the information gathered in Table $2.2$. The only case
that requires further discussion is the summands with vertex $2$ for $G=Co_3$ in part (c); this case follows from the next Proposition $2.12$ concerning projective summands of the reduced Lefschetz module for $M_{12}$.
\end{proof}

\begin{center}
{\small
\begin{tabular}{|c|c|l|l|l|}
\hline
$G$ & $Q$ & $N_G(Q)$ &$\Delta^Q$  & vertex of \\
 &  &  & up to homotopy & $\widetilde{L}_{N_G(Q)}(\Delta^Q; \mathbb{F}_p)$\\
\hline\hline
$M_{12}$ & $2$ & $2 \times S_5$ & $5$ points $= L_2(4)$ building & $2^2$ \\
         & $2^2$& $A_4 \times S_3$ & $3$ points $= L_2(2)$ building & $2^2=Q$ \\
\hline
$J_2$ & $2$ & $2^2 \times A_5$ & $5$ points $= L_2(4)$ building & $2^2$ \\
      & $2^2$ & $A_4 \times A_5$ & $5$ points $= L_2(4)$ building & $2^2=Q$ \\
      & $2^2$ & $2^4:3$ & contractible &  \\
\hline
$HS$ & $2$ & $2 \times A_6.2^2$ & $Sp_4(2)$ building & $2^2$ \\
     & $2^2$ & $2^2 \times 5:4$ & $Sz(2)$ building & $2^2=Q$ \\
\hline
$Ru$ & $2$ & $2^2 \times Sz(8)$ & $65$ points $= Sz(8)$ building & $2^2$ \\
     & $2^2$ & $(2^2 \times Sz(8)):3$ & $65$ points $= Sz(8)$ building & $2^2=Q$ \\
\hline
$Co_3$& $2$ & $2 \times M_{12}$ & geometry for $M_{12}$ & $2^3$ and $2=Q$ \\
      & $2^2$ & $A_4 \times S_5$ & $5$ points $= L_2(4)$ building & $2^3$ \\
      & $2^3$ & $(2^3 \times S_3).(7:3)$ & $3$ points $= L_2(2)$ building& $2^3=Q$ \\
\hline
$Suz$ & $2$ & $(2^2 \times L_3(4)).2$ & $L_3(4)$ building & $D_8$ \\
      & $2^2$ & $(A_4 \times L_3(4)).2$ & $L_3(4)$ building & $D_8$ \\
      & $2^2$ & $2^2 \times 2^{2+4}$ & contractible & \\
      & $2^2$ & $D_8 \times 3^2:Q_8$ & $9$ points $= PSU_3(2)$ building & $D_8$ \\
      & $4$ & $D_8 \times 3^2:Q_8$ & $9$ points $=PSU_3(2)$ building & $D_8$ \\
      & $D_8$ & $D_8 \times 3^2:Q_8$ & $9$ points $= PSU_3(2)$ building & $D_8=Q$ \\
      & $D_8^*$ & $D_8^* \times 2$ & contractible &  \\
\hline
\end{tabular}}\\

\vspace*{.4cm}
{\it Table 2.2: Fixed point sets and vertices of $\widetilde{L}_G(\Delta; \mathbb{F}_p)$}
\end{center}

\begin{prop} For the sporadic simple group $M_{12}$, the reduced Lefschetz module $\widetilde{L}_{M_{12}}(\Delta; \mathbb{F}_2)$ associated to the Benson complex has either one or two projective summands.
\end{prop}

\begin{proof} Recall that, by Lemma $2.2$ the Benson complex and the $2$-local geometry of $M_{12}$ are equivariantly homotopy equivalent; therefore they have equal Lefschetz modules. The geometry $\Delta$ for $M_{12}$ is a graph with two orbits of vertices, stabilized by $P_1=2^{1+4}_+:S_3$ and $P_2=4^2:D_{12}$. An edge is stabilized by a Sylow $2$-subgroup $P$ of $G$, a group of order $64$. Then
$$\widetilde{L}_{M_{12}}(\Delta; \mathbb{F}_2) = Ind_{P_1}^{M_{12}}(1) + Ind_{P_2}^{M_{12}}(1) - Ind_P^{M_{12}}(1)-1.$$
We will apply three results on projective summands of induced modules.

The following result is due to Robinson \cite[Theorem 3]{rob89} and Webb \cite[Proposition 5.3]{wb87b}:

\begin{lem}[Robinson, Webb] The number of summands in $Ind_H^G(1)$ isomorphic to the projective cover $P_G(S)$ of a simple $\mathbb{F}_pG$-module $S$ equals the number of summands in $Res_H^G(S)$ isomorphic to $P_H(1)$.
\end{lem}

If $P_G(S)$ is a summand of $Ind_H^G(1)$, then dim($S$) must at least be the cardinality of a Sylow $p$-subgroup of $H$. The simple $\mathbb{F}_2M_{12}$-modules have dimensions $1, 10, 16, 16, 44$ and $144$. Therefore the only possible projective summand of $\widetilde{L}_{M_{12}}(\Delta; \mathbb{F}_2)$ is the projective cover of $\varphi_6=144$. Note $Ind_P^{M_{12}}(1)$ contains at most two summands isomorphic to $P_{M_{12}}(\varphi _6)$.

The next Lemma, due to Robinson \cite[Proposition 6]{rob89}, will give an upper bound on the number of projective summands of $\widetilde{L}_{M_{12}}(\Delta; \mathbb{F}_2)$.

\begin{lem}[Robinson] Let $H \leq G$, $g \in G$ and let $\lbrace S_i: 1\leq i \leq t \rbrace$ be a full set of isomorphism types of simple $\mathbb{F}_pG$-modules. The number of double cosets $HgH$ for which $p$ does not divide $|H \cap H^g|$ is at least $\sum_{i,j=1}^tc_{ij}m_im_j$, where $(c_{ij})$ is the Cartan matrix of $\mathbb{F}_pG$, and $m_i$ denotes the multiplicity of $P_G(S_i)$ as a summand of $Ind_H^G(1)$.
\end{lem}

With $H=P_1$ or $H=P_2$, there are eleven double cosets $HgH$ in $M_{12}$, but for all of them $|H \cap H^g|$ is even (GAP \cite{gap} computation). Therefore $Ind_{P_1}^{M_{12}}(1)$ and $Ind_{P_2}^{M_{12}}(1)$ contain no projective summands. For $H=P \in {\rm Syl}_2(M_{12})$, there are $44$ double cosets $PgP$, and $12$ of them satisfy $|P \cap P^g|=1$. This implies that $2m^2 \leq 12$, where $m$ is the multiplicity of the projective cover $P_{M_{12}}(\varphi _6)$ as a summand of $Ind_{P}^{M_{12}}(1)$. Therefore $\widetilde{L}_{M_{12}}(\Delta; \mathbb{F}_2)$ contains at most two projective summands.

The third result used in this proof is due to Landrock \cite[Theorem 2.3]{lan86}; this will give a lower bound on the number of projective summands of $\widetilde{L}_{M_{12}}(\Delta; \mathbb{F}_2)$.

\begin{lem}[Landrock] Let $P \in {\rm Syl}_p(G)$. Then $Ind_P^G(1)$ is projective-free if and only if for every element $g \in G$ with $|P \cap P^g|=1$ and for every $p$-section $C$ (the set of elements of $G$ whose $p'$-part lies in some fixed $p$-regular conjugacy class), the number of $P$-orbits of elements in $C \cap (PgP)$ is a multiple of $p$.
\end{lem}

Applied to $G = M_{12}$, we find that $Ind_{P}^{M_{12}}(1)$ is not projective-free. A computation using GAP \cite{gap} provides an example of such a double coset and conjugacy class (of elements of order $11$) with five $P$-orbits.

Thus we have shown that the reduced Lefschetz module for $M_{12}$ has either one or two projective summands. This implies that the reduced Lefschetz module for $Co_3$ has either one or two summands with vertex $2$.
\end{proof}

\begin{rem}In a private communication, Klaus Lux has shown that the reduced Lefschetz module $\widetilde{L}_{M_{12}}(\Delta; \mathbb{F}_2)$ has precisely two projective summands. This implies that the reduced Lefschetz module for $Co_3$ has two summands with vertex $2$. Applying Lemma $2.13$ for $G=M_{12}$, $H \in {\rm Syl}_2(M_{12})$, and $S$ the simple $\mathbb{F}_2 M_{12}$-module of dimension $144$, we want to compute the number of projective summands in $Res_H^{M_{12}}(S)$. Observe that the only projective indecomposable $\mathbb{F}_2H$-module is the regular representation $P_H(1)=\mathbb{F}_2H$. The element $e = \sum _{h \in H}h$ lies in the socle of the group ring $\mathbb{F}_2H$, which acts trivially on any nonprojective indecomposable $\mathbb{F}_2H$-module. For the regular representation, each $h \in H$ corresponds to a permutation matrix and $e$ acts as a matrix with every entry equal to one. Therefore the rank of $e$ acting on an $\mathbb{F}_2H$-module equals the number of projective summands. Using GAP, Lux computes the rank of $e$ acting on $Res_H^{M_{12}}(S)$ to be two.

Using other techniques involving GAP and MeatAxe, Lux also shows that
$$\widetilde{L}_{M_{12}}(\Delta; \mathbb{F}_2) = S - 2 P_{M_{12}}(S),$$
where $S$ has Brauer character $\varphi_6=144$.
\end{rem}

\section{Sporadic simple groups with small Sylow $p$-subgroups. Case $p$ odd}

In this section we are concerned with the sporadic simple groups $G$ which have a Sylow $p$-subgroup of order $p^3$, for $p$ an odd prime. In all these cases the Sylow $p$-subgroup is $p^{1+2}_+$, extraspecial of exponent $p$. A concise treatment of these groups can be found in the paper of Ruiz and Viruel \cite{rv04} on fusion and linking systems.

\subsection{The distinguished Bouc complex}A $p$-subgroup $Q$ is $p$-{\it radical} if it is the largest normal $p$-subgroup in its normalizer $N_G(Q)$. In addition, $Q$ is called {\it distinguished} if it contains $p$-central elements in its center. The {\it distinguished Bouc collection} $\widehat{\mathcal{B}}_p(G)$ contains the distinguished $p$-radical subgroups of $G$. We shall denote by $|\widehat{\mathcal{B}}_p(G)|$ the {\it distinguished Bouc complex}.

We determine the nonprojective indecomposable summands, as well as their distribution into the blocks of the group ring $\mathbb{F}_pG$, of the reduced Lefschetz modules $\widetilde{L}_G(|\widehat{\mathcal{B}}_p(G)|; \mathbb{F}_p)=\widetilde{L}_G(\widehat{\mathcal{B}}_p) $ of these sporadic simple groups. For a $p$-element $t$ of noncentral type we first find the fixed point set $\Delta^t$ and then we apply Theorem $2.7$ (Robinson's formulation of the Burry-Carlson Theorem).

The groups $Ru$ and $J_4$ (with $p=3$) and $Th$ (with $p=5$) have one class of elements of order $p$ each. The groups $McL$ (with $p=5$) and $O'N$ (with $p=7$) have local characteristic $p$; all the $p$-local subgroups $H$ satisfy the condition $C_H(O_p(H))\leq O_p(H)$. In all these cases, the distinguished Bouc collection equals the whole Bouc collection and therefore the reduced Lefschetz modules are projective. For a proof of the equality of the two collections in a group of local characteristic $p$ see \cite[Lemma 4.8]{mgo3}. Information on the $p$-radical subgroups for odd primes, for $G$ one of the sporadic simple groups, is given in \cite[Table 1]{y05b}.

The remaining sporadic groups with Sylow $p$-subgroups of order $p^3$ are discussed below.

\begin{thm} Let $t \in G$ be an element of order $p$ of noncentral type.
\begin{list}{\upshape\bfseries}
{\setlength{\leftmargin}{.5cm}
\setlength{\rightmargin}{0cm}
\setlength{\labelwidth}{1cm}
\setlength{\labelsep}{0.2cm}
\setlength{\parsep}{0.5ex plus 0.2ex minus 0.1ex}
\setlength{\itemsep}{.6ex plus 0.2ex minus 0ex}}

\item[a)]Let $G$ be either $M_{12}$ or $J_2$, and let $p = 3$.  Then the fixed point set $|\widehat{\mathcal{B}}_3(G)|^t$ consists of four contractible components, homotopy equivalent to the building for $L_2(3)$. The reduced Lefschetz module $\widetilde{L}_G(\widehat{\mathcal{B}}_3)$ contains precisely one nonprojective summand, which has vertex $3 = \langle t \rangle$ and lies in a block with
the same group as defect group.
$$\widetilde{L}_{M_{12}}(\widehat{\mathcal{B}}_3) = P_{M_{12}}(\phi_2)-P_{M_{12}}(\phi_3)-P_{M_{12}}(\phi_6)+
P_{M_{12}}(\phi_8)-P_{M_{12}}(\phi_9)-\phi_7-\phi_{11}.$$
These projective covers all lie in the principal block; $\chi_{14}=\phi_7+\phi_{11}$ is not projective and lies in the block with defect group $\langle 3B \rangle = \langle t \rangle$. This formula is valid in the Grothendieck ring of characters.
$$\widetilde{L}_{J_2}(\widehat{\mathcal{B}}_3) = P_{J_2}(\phi_{10}) + \phi_9.$$
This formula is valid in the Green ring of virtual modules.

\item[b)] Let $G$ be $Ru, HS, Co_3$ or $Co_2$, and let $p=5$.
The group $HS$ has two classes of noncentral elements, denoted $5B$ and $5C$. When $G = HS$, let $t$ be a noncentral element of type $5B$; the fixed point set $|\widehat{\mathcal{B}}_5(HS)|^{5C}$ is contractible. Then the fixed point set $|\widehat{\mathcal{B}}_5(G)|^t$ consists of six contractible components, homotopy equivalent to the building for $L_2(5)$. The reduced Lefschetz module $\widetilde{L}_G(\widehat{\mathcal{B}}_5)$ contains precisely one nonprojective summand, which has vertex $5 = \langle t \rangle$ and lies in a block with the same group as defect group.

\item[c)] Let $G$ be $He$ and let $p=7$.  The simple group of Held has five conjugacy classes
of elements of order 7, which generate three classes of groups, denoted $7AB$, $7C$(of central
type), and $7DE$.  The fixed point set of the last group $7DE$ is contractible. Let $t$ denote
an element of type $7A$ or $7B$. Then the fixed point set $|\widehat{\mathcal{B}}_7(He)|^t$ consists of eight contractible
components, homotopy equivalent to the building for $L_2(7)$.  The virtual module $\widetilde{L}_{He}(\widehat{\mathcal{B}}_7)$ contains
precisely one nonprojective summand, which has vertex $7AB = \langle t \rangle$ and lies in a block with
the same group as defect group.

\item[d)] Let $G$ be either $He$ or $M_{24}$, and let $p=3$.  Then the fixed point set
$|\widehat{\mathcal{B}}_3(G)|^t$ consists of $28$ contractible components, equivalent to the set of Sylow
subgroups $Syl_3(L_3(2))$.  $\widetilde{L}_G(\widehat{\mathcal{B}}_3)$ contains three nonprojective
summands, all having vertex $3 = \langle t \rangle$.  Two of these summands lie in one block,
with the same group as defect group, but the third summand lies in the principal block.

\item[e)] Let $G$ be $Fi'_{24}$ and let $p=7$.  Then the fixed point set
$|\widehat{\mathcal{B}}_7(Fi'_{24})|^t$ consists of $120$ contractible components, equivalent to the set of Sylow
subgroups $Syl_{7}(A_7)$. The reduced Lefschetz module $\widetilde{L}_{Fi'_{24}}(\widehat{\mathcal{B}}_7)$ contains five nonprojective summands, all having vertex $7 = \langle t \rangle$.  These five summands lie in four
blocks; four lie in blocks with the same group $7 = \langle t \rangle$ as defect group, but the fifth summand lies in the principal block.

\item[f)] Let $G$ be the Monster $M$ and let $p=13$. Then the fixed point set
$|\widehat{\mathcal{B}}_{13}(M)|^t$ consists of $144$ contractible components, equivalent to the set of Sylow
subgroups $Syl_{13}(L_3(3))$.  The virtual module $\widetilde{L}_M(\widehat{\mathcal{B}}_{13})$ contains three nonprojective summands, all having vertex $13 = \langle t \rangle$.  These three summands lie in three
different blocks; two lie in blocks with the same group $13 = \langle t \rangle$ as defect group, but the third summand lies in the principal block.
\end{list}
\end{thm}

\begin{proof}
As mentioned above these sporadic groups, with the given prime, have Sylow $p$-subgroups equal
to the extraspecial group of order $p^3$ and exponent $p$.  Except for Held at $p=7$ and
$HS$ at $p=5$, they have two conjugacy classes of elements of order $p$ (usually type A are $p$-central and type B are noncentral, although this notation is reversed for $Fi'_{24}$ at $p=7$ and $M$ at $p=13$).  The set of
$p$-central elements is closed under taking products of commuting elements.  Thus the distinguished
Bouc collection $\widehat{\mathcal{B}}_p(G)$ is homotopy equivalent to the Benson collection,
which consists of nontrivial purely $p$-central elementary abelian subgroups \cite[Theorems 3.1 and 4.4]{mgo3}. By Proposition $2.8$, the fixed point set $|\widehat{\mathcal{B}}_p(G)|^z$ of a $p$-central element $z$ is contractible and $\widetilde{L}_G(\widehat{\mathcal{B}}_p)$ has no summands with a vertex containing a $p$-central element.  We need only consider purely noncentral $p$-subgroups, and for these Sylow $p$-subgroups $p_+^{1+2}$ this implies that the only possible nontrivial vertex is a group of order $p$ of noncentral type.

Usually, the first step in computing the fixed point set $|\widehat{\mathcal{B}}_p(G)|^t$ is to find the number
of Sylow $p$-subgroups $S$ which satisfy $t \in N_G(S)$.  In most of these cases,
$N_G(S)$ is equal to $N_G(Z)$, where $Z=Z(S)$ is a group of order $p$ of central type.  But $t \in N_G(Z)$ iff $t \in C_G(Z)$ iff $Z \subseteq C_G(t)$.  So we want to identify the p-central elements of $G$ that lie in $C_G(t)$.

$a)$. $\mathbf{M_{12}}$ and $\mathbf{p=3}$\\
A $3$-local geometry for $M_{12}$, originally due to Glauberman, appears in papers by Ronan and Stroth \cite[Section 3]{rstr84} and Buekenhout \cite[Section 9]{bue82} and $\widehat{\mathcal{B}}_3(M_{12})$ is its barycentric subdivision. The $3$-radical subgroups of $M_{12}$ form four conjugacy classes \cite[Section 2]{ac95}, given by $3= \langle t \rangle$, two conjugacy classes $3A^2_I$ and $3A^2_{II}$ of purely $3$-central elementary abelian $3$-subgroups of rank $2$, and the Sylow $3$-subgroup $3_+^{1+2}$. The last three are distinguished; thus $|\widehat{\mathcal{B}}_3(M_{12})|$ is a graph with three types of vertices. Given a Sylow $3$-subgroup $S$, let $Z=Z(S)$ be its center. Then $N_{M_{12}}(S)=N_{M_{12}}(Z)=3_+^{1+2}.2^2$ and $C_{M_{12}}(Z)=3_+^{1+2}.2$. Then $S \in |\widehat{\mathcal{B}}_3(M_{12})|^t$ iff $t \in N_{M_{12}}(S)$ iff $t \in C_{M_{12}}(Z)$ iff $Z \leq C_{M_{12}}(t)$. But $C_{M_{12}}(t) \simeq 3 \times A_4$ contains four $3$-central subgroups of order $3$ (eight elements of type $3A$) \cite[Section 3]{wh66}, and so $|\widehat{\mathcal{B}}_3(M_{12})|^t$ contains four vertices of type $3_+^{1+2}$. The element $t$ will also normalize each $3A^2 \leq 3_+^{1+2}$ (the Sylow group contains two purely $3$-central $3A^2$), but two ``adjacent" Sylows will have an intersection $S_1 \cap S_2 = 3A^2$ containing no elements of type $3B$. If $3A^2 \in |\widehat{\mathcal{B}}_3(M_{12})|^t$ then $t \in N_{M_{12}}(3A^2)$ and $\langle t,3A^2 \rangle = 3_+^{1+2} \in |\widehat{\mathcal{B}}_3(M_{12})|^t$. Thus $|\widehat{\mathcal{B}}_3(M_{12})|^t$ has four components, each component equal to $2$ edges and $3$ vertices.

The normalizer $N_{M_{12}}(\langle t \rangle) \simeq S_3 \times A_4$ acts on the fixed point set $|\widehat{\mathcal{B}}_3(M_{12})|^t$ with the $S_3$ acting trivially and so that the action of $A_4$ is equivalent to the action of $L_2(3)$ on its building. Thus $\widetilde{L}_{A_4}(|\widehat{\mathcal{B}}_3(M_{12})|^t)$ is the Steinberg module (projective and irreducible), and $\widetilde{L}_{S_3 \times A_4}(|\widehat{\mathcal{B}}_3(M_{12})| ^t)$ is indecomposable with vertex $3 = \langle t \rangle$. This implies that the Lefschetz module $\widetilde{L}_{M_{12}}(\widehat{\mathcal{B}}_3)$ contains one indecomposable summand with vertex of type $\langle 3B \rangle$, and it lies in a block with the same group $\langle 3B \rangle$ as defect group.

Let $H_1$ and $H_2$ denote the two conjugacy classes of $3^2:GL_2(3)$ in $M_{12}$ (the normalizers of the two classes of subgroups $3A_I^2$ and $3A_{II}^2$), with intersection $H_1 \cap H_2 = H_{12} = 3^{1+2}_+:2^2$.  Then the reduced Lefschetz module is:
$$\widetilde{L}_{M_{12}}(\widehat{\mathcal{B}}_3)= Ind_{H_1}^{M_{12}}(1a)+ Ind_{H_2}^{M_{12}}(1a)-Ind_{H_{12}}^{M_{12}}(1a)-1a.$$
The characters for the first two terms are described in the Atlas; the third was computed with GAP,
using
$$Ind_{H_1}^{M_{12}} Ind_{H_{12}}^{H_1}(1a) = Ind_{H_1}^{M_{12}}(1a + 3b).$$
The character formula given above in Theorem $3.1(a)$ was obtained. Observe that the part of this formula that lies
in the principal block, being projective, is valid in the Green ring of virtual modules. There are only six
indecomposable modules lying in the block with the cyclic defect group $\langle 3B \rangle$: the
two simple modules $\phi_7$ and $\phi_{11}$; their projective covers, which are uniserial,
$P_{M_{12}}(\phi_7)= \langle \phi_7, \phi_{11}, \phi_7 \rangle$ and $P_{M_{12}}(\phi_{11})= \langle \phi_{11}, \phi_7, \phi_{11} \rangle$; and two other modules $M_1 = \langle \phi_7, \phi_{11} \rangle$ and $M_2 = \langle \phi_{11}, \phi_7 \rangle$. Unfortunately, $M_1$ and $M_2$ have the same factors and the same character. One of these must be the nonprojective indecomposable summand of the reduced Lefschetz module.

\smallskip
$\mathbf{J_2}$ and $\mathbf{p=3}$\\
The Hall-Janko-Wales group $J_2$ has two conjugacy classes of elements of order $3$, denoted $3A$ and $3B$. Their normalizers are $N_{J_2}(\langle 3A \rangle ) = 3.PGL_2(9) = 3.A_6.2$ and $N_{J_2}(\langle 3B \rangle )=S_3 \times A_4$. The Sylow subgroup contains two elements of type $3A$ (the center) and $24$ elements of type $3B$. The collection of elements of type $3A$ is closed (a GAP computation gives that the class multiplication coefficient $\xi(3A, 3A, 3B)=0$) and thus the Benson complex $\Delta$ consists of a disjoint set of $280$ vertices corresponding to the groups of type $\langle 3A \rangle$, and this Benson complex is equivariantly homotopy equivalent to $|\widehat{\mathcal{B}}_3(J_2)|$. To show that the fixed point set $\Delta^{3B}$ consists of four vertices, note that $3B \in N_{J_2}(\langle 3A \rangle)$ iff $3B \in C_{J_2} (3A)$ iff
$3A \in C_{J_2} (3B)=3 \times A_4$. The group $3 \times A_4$ contains four subgroups of type $\langle 3A \rangle$, namely ${\rm Syl}_3(A_4)$. These computations imply that the reduced Lefschetz module $\widetilde{L}_{J_2}(\widehat{\mathcal{B}}_3)$ has precisely one indecomposable summand with vertex $3= \langle 3B \rangle$ which lies in a block with the same group as defect group. Note that
$$\widetilde{L}_{J_2}(\widehat{\mathcal{B}}_3) = Ind_{3.A_6.2}^{J_2}(1a) - 1a = 63a+90a+126a =
\chi_7+\chi_{10}+\chi_{11}= \phi_9+P_{J_2}(\phi_{10})$$
is an ordinary module (not a virtual module). The module which affords $\chi_7=\phi_9$ is not projective, and lies in block 3 with defect group $\langle 3B \rangle$. The projective cover $P_{J_2}(\phi_{10})=\chi_{10}+\chi_{11}$ lies in block 2 with defect group $\langle 3A \rangle$. Since $\langle 3A \rangle$ is not a vertex, as $|\widehat{\mathcal{B}}_3(J_2)|^{3A}$ is contractible, the part of the reduced Lefschetz module lying in block 2
must be projective.  The entire formula given for $\widetilde{L}_{J_2}(\widehat{\mathcal{B}}_3)$ is valid on the
level of modules, not just characters.

\smallskip
$b)$. $\mathbf{Ru}$ and $\mathbf{p=5}$\\
The sporadic simple group of Rudvalis has two conjugacy classes of elements of order $5$, the $5$-central elements $5A$ and the noncentral elements $5B$. There are three conjugacy classes of radical $5$-subgroups, $5=\langle 5B \rangle$, a purely central $5^2 = 5A^2$, and the Sylow $5_+^{1+2}$; see \cite[Table 4]{ab02}. Therefore $|\widehat{\mathcal{B}}_5(Ru)|$ is a graph with two types of vertices. To compute the fixed point set $|\widehat{\mathcal{B}}_5(Ru)|^{5B}$, let $S = 5_+^{1+2} \in |\widehat{\mathcal{B}}_5(Ru)|^{5B}$ and denote by $Z$ the center $Z(S) = \langle 5A \rangle$.
Then $5B \in N_{Ru}(S) = N_{Ru}(Z) \simeq 5_+^{1+2} : ((4 \times 4 ) :2)$ iff
$Z = \langle 5A \rangle \leq  C_{Ru}(5B) \simeq 5 \times A_5$. The alternating group $A_5$ contains six subgroups of order $5$,
all of type $\langle 5A \rangle$. Thus $|\widehat{\mathcal{B}}_5(Ru)|^{5B}$ contains the six vertices corresponding to six Sylow
$5$-subgroups of $Ru$. If $5B \in S$, then $5B \in N_{Ru}(5A^2)$ for the two copies of $5A^2 \leq S$.
Conversely, if $5B \in N_{Ru}(5A^2)$, then $\langle 5B, 5A^2 \rangle$ is a Sylow $5$-subgroup of $Ru$.
Therefore $|\widehat{\mathcal{B}}_5(Ru)|^{5B}$ consists of six components, each consisting of two edges with their three vertices.

In the action of $N_{Ru}(\langle 5B \rangle) \simeq (5:4) \times A_5$ on $|\widehat{\mathcal{B}}_5(Ru)|^{5B}$, the Frobenius group $5:4$ acts trivially and the action of $A_5$ is equivalent to the action of $L_2(5)$ on its building. Thus the reduced Lefschetz module $\widetilde{L}_{A_5}(|\widehat{\mathcal{B}}_5(Ru)|^{5B})$ is the projective irreducible Steinberg module, and $\widetilde{L}_{N_{Ru}(\langle 5B \rangle)}(|\widehat{\mathcal{B}}_5(Ru)|^{5B})$ is indecomposable with vertex
$5= \langle 5B \rangle$. Therefore $\widetilde{L}_{Ru}(\widehat{\mathcal{B}}_5)$ contains one indecomposable summand with vertex $\langle 5B \rangle$, and it lies in a block with the same group as defect group.

\smallskip
$\mathbf{G = HS, Co_3, Co_2}$ and $\mathbf{p=5}$\\
These three groups have only two conjugacy classes of radical $5$-subgroups, $\langle 5B \rangle$ and the Sylow $5$-subgroup $5^{1+2}_+$. Thus $|\widehat{\mathcal{B}}_5(G)|$ is a disjoint set, with vertices corresponding to ${\rm Syl}_5(G)$. For $G = HS$ or $G=Co_3$, $N_G (\langle 5B \rangle) \simeq (5:4) \times A_5$, but for $G=Co_2$ we have $N_G (\langle 5B \rangle) \simeq (5:4) \times S_5$. The fixed point set $|\widehat{\mathcal{B}}_5(G)|^{5B}$ consists of six vertices, corresponding to those groups $S \in {\rm Syl}_5(G)$ with centers $\langle 5A \rangle = Z(S)$ lying in the $A_5$ (or $S_5$) of $N_G (\langle 5B \rangle)$. Thus $\widetilde{L}_{N_G (\langle 5B \rangle)}(|\widehat{\mathcal{B}}_5(G)| ^{5B})$ is
indecomposable with vertex $5= \langle 5B \rangle$. (Note the Steinberg module for $A_5$ extends in a natural
way to one for $S_5$ \cite{schmid92}, although $S_5$ has two projective modules which restrict to the Steinberg
module for $A_5$.) Therefore $\widetilde{L}_G(\widehat{\mathcal{B}}_5)$ contains one indecomposable summand with vertex $\langle 5B \rangle$, lying in a block with the same group as defect group.

Observe that $HS$ has another class $5C$ of noncentral elements; the normalizer $N_{HS}(\langle 5C \rangle) = 5^2:4$. The elementary abelian $5^2$ contains a unique central subgroup $\langle 5A \rangle$.  This implies that the fixed point set $|\widehat{\mathcal{B}}_5(HS)|^{5C}$ is contractible, and $\langle 5C \rangle$ is not a vertex of the Lefschetz module.

\smallskip
$c)$. $\mathbf{He}$ and $\mathbf{p=7}$\\
The sporadic group of Held contains five conjugacy classes of elements of order $7$, with $7C$
the $7$-central elements; the $7B$ are inverses of the $7A$ and the $7E$ are inverses of the
$7D$. The normalizers are $N_{He}(7AB)=(7:3) \times L_2(7)$, $N_{He}(7DE)=7^2.6$, and
$N_{He}(\langle 7C \rangle)=N_{He}(7^{1+2}_+)=7^{1+2}_+:(S_3 \times 3)$. There is one class of
purely central elementary abelian $7C^2$, as well as two mixed classes of elementary abelian
$7^2$, each containing six $7$-central elements; see \cite[Section 2.4]{but81}. There are three classes of $7$-radical subgroups, $7AB$, $7C^2$, and $7^{1+2}_+$; the last two are distinguished \cite[Section 2]{an97}. Thus $|\widehat{\mathcal{B}}_7(He)|$ is a graph with two types of vertices. Each Sylow $7$-subgroup contains three copies of $7C^2$, and each elementary abelian $7C^2$ lies in eight Sylow subgroups. Note $7AB \leq N_{He}(7^{1+2}_+)=N_{He}(\langle 7C \rangle)$ iff $7C \in C_{He}(7AB)=7 \times L_2(7)$. Thus the fixed point set $|\widehat{\mathcal{B}}_7(He)|^{7AB}$ contains eight vertices corresponding to Sylow subgroups of Held. Such a Sylow will contain three copies of $7C^2$ also normalized by the $7AB$. If $7AB \leq N_{He}(7C^2)$, then they generate a Sylow subgroup $7^{1+2}_+=\langle 7AB, 7C^2 \rangle$. The fixed point set $|\widehat{\mathcal{B}}_7(He)|^{7AB}$ consists of eight contractible components, and $\widetilde{L}_{He}(\widehat{\mathcal{B}}_7)$ has a summand with vertex $\langle 7AB \rangle$, lying in a block with the same group as the defect group.

If $7DE \leq N_{He}(7^{1+2}_+)=N_{He}(\langle 7C \rangle)$, then $7C \in C_{He}(7DE)$. But the latter contains
only one copy of $\langle 7C \rangle$. If $7DE \leq N_{He}(7C^2)$, they generate a Sylow subgroup
$7^{1+2}_+=\langle 7DE, 7C^2 \rangle$. Thus the fixed point set $\Delta^{7DE}$ is contractible, consisting of the star corresponding to one Sylow subgroup and its three copies of $7C^2$. $\widetilde{L}_{He}(\widehat{\mathcal{B}}_7)$ does not contain any summand with vertex $\langle 7DE \rangle$.

\smallskip
$d)$. $\mathbf{He}$ and $\mathbf{p=3}$\\
The sporadic group of Held contains two conjugacy classes of elements of order $3$, denoted $3A$ and $3B$. The Sylow $3$-subgroups $3_+^{1+2}$ contain $14$ elements of type $3A$ and $12$ elements of type $3B$. There are two
conjugacy classes of elementary abelian $3$-subgroups of rank $2$, one of type $3A^2$ (a pure
group, with $8$ elements of type $3A$; note $3_+^{1+2}$ contains two such subgroups) and one mixed $3^2 = A_1B_3$
(containing $2$ elements of type $3A$ and $6$ elements of type $3B$; note $3_+^{1+2}$ contains two of these groups); see \cite[Section 2.2]{but81}. The nontrivial $3$-radical subgroups of Held \cite[Section 2]{an97} are the conjugacy classes of $\langle 3A \rangle, \langle 3B \rangle$, $3A^2$, and $3_+^{1+2}$. Thus the distinguished Bouc complex $|\widehat{\mathcal{B}}_3(He)|$ contains three types of vertices, corresponding to $\langle 3A \rangle$, $3A^2$, and $3_+^{1+2}$; note $|\widehat{\mathcal{B}}_3(He)|$ is a $2$-dimensional complex with maximal simplices corresponding to chains $\langle 3A \rangle \leq 3A^2 \leq 3_+^{1+2}$. We wish to compute the
fixed point set $|\widehat{\mathcal{B}}_3(He)|^{3B}$. Clearly $3B \in N_{He}(\langle 3A \rangle)$ iff $3B \in
C_{He}(3A)$ iff $3A \in C_{He}(3B)= 3 \times L_3(2)$. There are precisely $28$ subgroups of type $\langle
3A \rangle$ in this group, namely ${\rm Syl}_3(L_3(2))$. Next, $3B \in N_{He}(3_+^{1+2})$ iff $3B \in 3_+^{1+2} \leq C_{He}(3A) = 3.A_7$, where $3A=Z(3^{1+2}_+)$. An element of type $3B$ lying in $C_{He}(3A)$ has image in $A_7$ a permutation in the conjugacy class of $(123)(456)$. Such a permutation is contained in a unique Sylow $3$-subgroup of $A_7$. Thus there is a unique Sylow subgroup $S = 3_+^{1+2}$ in $He$ with center $\langle 3A \rangle$ and containing $3B$. So $|\widehat{\mathcal{B}}_3(He)|^{3B}$ contains $28$ edges connecting vertices of type $\langle 3A \rangle$ to vertices of type $3_+^{1+2}$. Clearly, if $3B \in 3_+^{1+2}$, then $3B$ normalizes the two elementary abelian subgroups of type $3A^2 \leq 3_+^{1+2}$. But if $3B \in N_{He}(3A^2)$, then they generate a Sylow, $\langle 3B, 3A^2 \rangle = 3_+^{1+2}$. Thus $|\widehat{\mathcal{B}}_3(He)|^{3B}$ contains $56$ vertices of type $3A^2$, and consists of $28$ components, each (contractible) component equal to two triangles sharing one common edge.

\begin{rem} Although the multiplication class coefficient $\xi(3A, 3A, 3B)= 168$ is nonzero,
it is still true that if two elements of type $3A$ commute with each other, then their product
either equals the identity or is another element of type $3A$. If the two commuting elements
generate an elementary abelian $3^2$, then this $3^2$ contains at least four elements of type
$3A$ (the given generators and their inverses). But this forces the $3^2$ to be a pure $3A^2$.
(The number $168$ can be explained as follows: in a fixed Sylow $3_+^{1+2}$ with a fixed
element $3B \in 3_+^{1+2}$, there are six ways to multiply two (noncommuting) elements of type
$3A$ to obtain the given element $3B$. Then each $3B$ is contained in $28$ Sylow subgroups, and
$168=6 \times 28$.) Thus the Benson complex $\Delta (He)$ equals the subcomplex of the
distinguished Bouc complex $|\widehat{\mathcal{B}}_3(He)|$ consisting of those terms
corresponding to the groups of type $\langle 3A \rangle$ and $3A^2$. The inclusion
$\Delta (He) \subseteq \widehat{\mathcal{B}}_3(He)$ induces an equivariant homotopy equivalence.
\end{rem}

The action of $L_3(2)$ on its $28$ Sylow $3$-subgroups yields a $28$-dimensional representation
$\chi = \chi _1 + 2 \chi _4 + \chi _5 + \chi _6$ (using the Atlas \cite{Atlas} notation, $\chi _1=1a$, $\chi
_4 = 6a$, $\chi _5 = 7a$, and $\chi _6 = 8a$). Note that
$N_{L_3(2)}(3)=S_3$, so the elements of order $4$ or $7$ have no fixed
points. The elements of order three fix precisely one subgroup $3$, and each element
of order two fixes four groups of order $3$. The Green correspondence implies that $Ind _{S_3}^{L_3(2)}(1)$ is a direct sum of the trivial module and a projective module. Therefore the reduced Lefschetz module
$\widetilde{L}_{L_3(2)}(|\widehat{\mathcal{B}}_3(He)|^{3B})$ is projective, with character $2 \chi _4 + \chi _5 +
\chi _6$. Note that $\chi _4 = \varphi _4$ is projective, lying in a block of defect zero. But
$\chi _5 + \chi _6$ is the projective cover $P_{L_3(2)}(\varphi _5)$, where $\chi _5 = \varphi
_5$ and $\chi _6 = \varphi _1 + \varphi _5$, which lies in the principal block. Note that the
$S_3$ in $N_{He}(\langle 3B \rangle ) \simeq S_3 \times L_3(2)$ acts trivially on $|\widehat{\mathcal{B}}_3(He)| ^{3B}$. It follows that $\widetilde{L}_{N_{He}(\langle 3B \rangle)}(|\widehat{\mathcal{B}}_3(He)|^{3B})$ has three summands with vertex $\langle 3B \rangle$. Therefore the reduced Lefschetz module $\widetilde{L}_{He}(\widehat{\mathcal{B}}_3)$ has precisely three nonprojective summands, all with vertex $3= \langle 3B \rangle$. Two of these summands lie in
the same block, with defect group also equal to $\langle 3B \rangle$, but the third summand
lies in the principal block.

\smallskip
$\mathbf{M_{24}}$ and $\mathbf{p=3}$\\
The situation for $M_{24}$ is very similar to that of $He$, with the same Sylow $3$-subgroup
and fusion pattern. Also, $N_{M_{24}}(\langle 3B \rangle) \simeq S_3 \times L_3(2)$ acts on
$|\widehat{\mathcal{B}}_3(M_{24})|^{3B}$ as $L_3(2)$ acts on its $28$ Sylow $3$-subgroups. Thus the reduced Lefschetz module $\widetilde{L}_{M_{24}}(\widehat{\mathcal{B}}_3)$ contains three indecomposable summands with vertex $3 = \langle 3B \rangle$. Two of these summands lie in the same block, having defect group $\langle
3B \rangle$. The third summand lies in the principal block.

\smallskip
$e)$. $\mathbf{Fi'_{24}}$ and $\mathbf{p=7}$\\
The Fischer group $Fi'_{24}$ has two classes of elements of order $7$, type $7A$ and $7B$ (the
latter are $7$-central).  The normalizers are $N_G(\langle 7A \rangle)=(7:6) \times A_7$ and
$N_G(\langle 7B \rangle)=N_G(7^{1+2}_+)=7^{1+2}_+:(S_3 \times 6)$. There are four conjugacy
classes of $7$-radical subgroups, $\langle 7A \rangle$, $7^{1+2}_+$, and two classes of purely central
elementary abelian $7B^2$.  The last three are distinguished. Thus $|\widehat{\mathcal{B}}_7(Fi'_{24})|$ is a graph with three types of vertices, corresponding to the Sylow subgroups and the two classes of
elementary abelian $7B^2$. The fixed point set $|\widehat{\mathcal{B}}_7(Fi'_{24})|^{7A}$ consists of $120$ contractible components, each component a star corresponding to one Sylow subgroup $7^{1+2}_+$ and its six purely central elementary abelian $7B^2$, three from each conjugacy class. These components correspond to
the $120$ Sylow $7$-subgroups of the alternating group $A_7$. The action of $A_7$ on ${\rm Syl}_7(A_7)$
yields the induced module
$$Ind_{7:3}^{A_7}(1a)=1a+10ab+14bb+15a+21a+35a$$
(the character values are $120,0,0,6,0,0,0,1,1$; the entry $6$, for example, can be thought of in
terms of the fact that the permutation $(123)(456)$ normalizes six different groups of order $7$ in $A_7$). The group ring of $A_7$ has five blocks at the prime $p=7$, with defects $1,0,0,0,0$. The module which affords $\chi_6 = \phi_5 =14b$ is projective and lies in block $3$, $\chi_8 = \phi_6 =21a$ is projective and lies in block $4$, and $\chi_9 = \phi_7 = 35a$ is projective and lies in block $5$. The projective cover $P_{A_7}(\phi_3)=
10ab+15a=\chi_3+\chi_4+ \chi_7$ lies in the principal block. This implies that the reduced Lefschetz module
$\widetilde{L}_{Fi'_{24}}(\widehat{\mathcal{B}}_7)$ contains five indecomposable summands with vertex $7 = \langle 7A \rangle$. Four of these summands lie in three blocks, having the same group as defect group (two summands are in the same block). The fifth summand lies in the principal block.

\smallskip
$f)$. $\mathbf{M}$ and $\mathbf{p=13}$\\
The Monster group $M$ has two classes of elements of order $13$, type $13A$ and $13B$ (the latter are $13$-central). The normalizers are $N_M(\langle 13A \rangle)= (13:6 \times L_3(3)).2$ and $N_M(\langle 13B \rangle)=N_M(13^{1+2}_+)=13^{1+2}_+:(3 \times 4.S_4)$. There is one class of purely central elementary abelian $13B^2$. The groups $\langle 13A \rangle$, $13B^2$, and $13^{1+2}_+$ are $13$-radical, and the last two are distinguished. Thus $|\widehat{\mathcal{B}}_{13}(M)|$ is a graph with two types of vertices; each Sylow contains six copies of $13B^2$, and each elementary abelian $13B^2$ lies in $14$ Sylow subgroups. Note that $13A \in N_M(13^{1+2}_+)=N_M(\langle 13B \rangle)$ iff $13B \in C_M(13A)=13 \times L_3(3)$. Thus the fixed point set $|\widehat{\mathcal{B}}_{13}(M)|^{13A}$ contains $144$ vertices corresponding to Sylow subgroups of $M$. For each such Sylow, the $13A$ also normalizes its six elementary abelian $13B^2$, and if $13A \in N_M(13B^2)$, then they generate a Sylow $13^{1+2}_+ = \langle 13A, 13B^2 \rangle$.

The group ring of $L_3(3)$ has six blocks at the prime $13$, of defects $1,0,0,0,0,0$. The action of $L_3(3)$ on its Sylow $13$-subgroups corresponds to the induced character
$$Ind_{13:3}^{L_3(3)}(1a)= 1a+13a+16abcd+27a+39a.$$
Note that $13a$ is projective and lies in block two, and $39a$ is projective and lies in block six. The projective cover $P_M(\phi_4)=16abcd+27a$ lies in the principal block. This implies that the reduced Lefschetz module $\widetilde{L}_{M}(\widehat{\mathcal{B}}_{13})$ contains three indecomposable summands with vertex $13 = \langle 13A \rangle$. Two of these summands lie in two blocks, having the same group as defect group.
The third summand lies in the principal block.
\end{proof}

\subsection{The complex of $\mathbf p$-centric and $\mathbf p$-radical subgroups}A $p$-subgroup $Q$ of $G$ $p$-{\it centric} if the center $Z(Q)$ is a Sylow $p$-subgroup of $C_G(Q)$. The collection of $p$-centric and $p$-radical subgroups $\mathcal{B}^{\rm{cen}}_p(G)$ is a subcollection of the distinguished Bouc collection $\widehat{\mathcal{B}}_p(G)$; see \cite[Proposition 3.1]{mgo4}.

A group has {\it parabolic characteristic} $p$ if all $p$-local subgroups which contain a Sylow $p$-subgroup of $G$ have characteristic $p$. For the groups with this property, the two collections $\widehat{\mathcal{B}}_p(G)$ and $\mathcal{B}^{\rm{cen}}_p(G)$ are equal \cite[Proposition 3.5]{mgo4}. All but three of the sporadic groups discussed in Section $3.1$ satisfy this condition. We now turn to the three cases above in which the complex of $p$-radical $p$-centric subgroups is not equal to the complex of distinguished $p$-radical subgroups.

\begin{thm}
Let $\Delta = |\mathcal{B}_3^{cen}(G)|$ and let $\widetilde{L}_G(\Delta)$ be the mod-$3$ reduced
Lefschetz module. Let $z \in G$ be a $3$-central element and $t \in G$ be an element of order
$3$ of noncentral type.

\begin{list}{\upshape\bfseries}
{\setlength{\leftmargin}{.5cm}
\setlength{\rightmargin}{0cm}
\setlength{\labelwidth}{1cm}
\setlength{\labelsep}{0.2cm}
\setlength{\parsep}{0.5ex plus 0.2ex minus 0.1ex}
\setlength{\itemsep}{.6ex plus 0.2ex minus 0ex}}
\item[a)] Let $G$ be $J_2$. Then the fixed point set $\Delta^z$ consists of ten vertices, equal to the building for $PGL_2(9)$, and the fixed point set $\Delta^t$ consists of four vertices, equal to the building for $L_2(3)$.\\
The module $\widetilde{L}_{J_2}(\Delta)$ contains two nonprojective summands. One summand has vertex $3 = \langle t \rangle$ and lies in a block with the same group as defect group. The second summand has vertex the $3$-central group $\langle z \rangle $ and lies in a block with the same group as defect group.

\item[b)] Let $G$ be $M_{24}$. Then the fixed point set $\Delta^z$ consists of ten contractible components, equivalent to $Syl_3(S_6)$, and the fixed point set $\Delta^t$ consists of $28$ contractible components, equivalent to the set of Sylow $3$-subgroups $Syl_3(L_3(2))$.\\
The module $\widetilde{L}_{M_{24}}(\Delta)$ contains four nonprojective
summands. Three of these summands have vertex $3 = \langle t \rangle$; two summands lie in one
block, with the same group as defect group, but the third summand lies in the principal block. The fourth summand has vertex the $3$-central group $ \langle z \rangle $, and lies in a block with the same group as defect group.

\item[c)] Let $G$ be $He$. Then the fixed point set $\Delta^z$ is connected and is homotopy equivalent to a subcomplex of the Quillen complex for the symmetric group $S_7$. The fixed point set $\Delta^t$ consists of $28$ contractible components, equivalent to the set of Sylow subgroups $Syl_3(L_3(2))$.\\
The module $\widetilde{L}_{He}(\Delta)$ contains four nonprojective summands. Three of these summands have vertex $3 = \langle t \rangle$;  two summands lie in one block, with the same group as defect group, but the third summand lies in the principal block. The fourth summand has vertex the $3$-central group $ \langle z \rangle $, and lies in a block with defect group the purely central elementary abelian group $3A^2$.
\end{list}
\end{thm}

\begin{proof}
Since $\Delta$ is a subcomplex of $|\widehat{\mathcal{B}}_3(G)|$ we shall use some of the details given in the proof of Theorem $3.1 (a,d)$. For these three groups, the $3$-central group $\langle 3A \rangle$ is
$3$-radical but not $3$-centric. For $J_2$ the complex $\Delta$ is a discrete set of vertices,
corresponding to ${\rm Syl}_3(J_2)$. For $M_{24}$ and $He$, the complex $\Delta$ is a graph with two types of
vertices, corresponding to the Sylow 3-subgroups and to the purely central elementary abelian groups
$3A^2$. In all three cases, the fixed point set $\Delta^t$ is equivalent to that described in the
previous theorem, with the same explanation for the summands of the reduced Lefschetz module
which have vertex the group $\langle t \rangle$. In the previous theorem, the fixed point
sets for the $3$-central element $z$ were contractible; this is no longer true for the complex
$\Delta$.

For $G = J_2$, the fixed point set $\Delta^z$ consists of ten vertices corresponding to the
ten Sylow $3$-subgroups containing $z$. The action of $N_{J_2}(\langle z \rangle)=3.PGL_2(9)$ on these ten
vertices is the action of $PGL_2(9)$ on its building, with reduced Lefschetz module the Steinberg
module. Then Theorem $2.7$ implies that $\widetilde{L}_{J_2}(\Delta)$ has one summand with vertex
$\langle z \rangle$ which lies in a block with the same group as defect group. This concludes the proof of the statement in part $a)$ of the theorem.

For $G$ equal to either $M_{24}$ or $He$, note that the collection consisting of the Sylow $3$-subgroups and the purely central elementary abelian subgroups $3A^2$ is closed under passage to $3$-supergroups.
We have an equivariant poset map $F_1:\Delta^z \rightarrow \Delta^z$ defined by $F_1(P)= \langle z \rangle \cdot P$. Since this satisfies $F_1(P) \geq P$, there is an $N_G(\langle z \rangle )$-homotopy equivalence between $\Delta^z$ and the image of $F_1$, the subcomplex $\Delta^z_1$ determined by those groups $P$ which contain $z$; see \cite[2.2(3)]{gs}. Next, $F_2:\Delta^z_1 \rightarrow \Delta^z_1$ defined by $F_2(P)=P \cap C_G(z)$ is an equivariant poset map satisfying $F_2(P) \leq P$. Observe that if $z \in P = 3A^2$ then $P \leq C_G(z)$,
and if $z \in P = 3^{1+2}_+$ then either $z \in Z(P)$ and $P \leq C_G(z)$, or $z$ is contained in only one
of the two purely central elementary abelian $3A^2$ in $P$.  In this latter case, $F_2(P)=3A^2$ does lie in $\Delta^z_1$.  Therefore $\Delta^z_1$ is $N_G(\langle z \rangle )$-homotopy equivalent to the image of $F_2$, the subcomplex $\Delta^z_2$ determined by those groups $P$ satisfying $z \in P \leq C_G(z)$; of course $P$ is either $3A^2$ or $3^{1+2}_+$.

The centralizers of the $3$-central elements are $C_{M_{24}}(z)=3.A_6$ and $C_{He}(z)=3.A_7$. Let $H$ denote $A_6$ when $G=M_{24}$ and $A_7$ when $G=He$.  The quotient map $\pi:3.H \rightarrow H$ induces an isomorphism of the complex $\Delta^z_2$ with the subcomplex $\mathcal{E}_3(H)$ of the Quillen complex for H determined by the collection consisting of the Sylow $3$-subgroups of $H$ (elementary abelian of rank 2) and the groups of order three generated by elements in the conjugacy class of the cycle $(123)$. Since an element of $H$ in the conjugacy class of the product $(123)(456)$ of two cycles lies in a unique Sylow $3$-subgroup of $H$, the subcomplex $\mathcal{E}_3(H)$ is equivariantly homotopy equivalent to the entire Quillen complex $\mathcal{A}_3(H)$. Recall that the reduced Lefschetz module for the Quillen complex of any finite group is projective.

For $G$ equal to either $M_{24}$ or $He$, the normalizers of the $3$-central elements are $N_{M_{24}}(\langle z \rangle )=3.S_6$ and $N_{He}(\langle z \rangle )=3.S_7$.
For $S_6$, the Sylow $3$-subgroups satisfy the trivial intersection property; in $M_{24}$, each $3A^2$ lies in a unique $S \in {\rm Syl}_3(M_{24})$. The induced module
$$Ind_{3^2:D_8}^{S_6}(1a)-1a=9a=\varphi_5$$
yields a projective irreducible module lying in a defect zero block of $\mathbb{F}_3S_6$. Consequently, the reduced Lefschetz module $\widetilde{L}_{M_{24}}(\Delta)$ contains a summand with vertex $3=\langle z \rangle$. This proves part $b)$ of the theorem.

For the symmetric group $S_7$ we have
$$Ind_{S_3 \times S_4}^{S_7}(1a) + Ind_{(S_3 \times S_3):2}^{S_7}(1a) - Ind_{S_3 \times S_3}^{S_7}(1a) - 1a = -(21a+15a)= -(\chi_8+\chi_{12})= -P_{S_7}(\varphi_2)$$
is projective and indecomposable and lies in a block of defect one. Thus $\widetilde{L}_{He}(\Delta)$ has a summand with vertex $3^2$, and this concludes the proof of part $c)$ of the theorem.
\end{proof}

\subsection{One last example}We conclude with one further example; although the Sylow subgroup is not extraspecial, the complex of distinguished $3$-radical subgroups is very simple (a disjoint collection of vertices).

\begin{thm}
Let $G = J_3$ and $p=3$.  Then the reduced Lefschetz module for the collection of
distinguished $3$-radical subgroups has precisely one nonprojective summand with vertex a
group of order $3$ (of noncentral type) and lying in a block with the same group as
defect group.
\end{thm}

\begin{proof}
The third Janko group $J_3$ has a Sylow $3$-subgroup $S$ of order $3^5$, and contains two conjugacy classes of
elements of order $3$, denoted $3A$ and $3B$; the class $3B$ lies in the center of the Sylow subgroup $S$. The center $Z(S) = 3B^2$ is an elementary abelian $3$-group of rank $2$ containing $8$ elements of type $3B$. The normalizers are $N_{J_3}(\langle 3A \rangle)=3:PGL_2(9)$, with $C_{J_3}(3A)= 3 \times A_6$, and $N_{J_3}(S)=N_{J_3}(3B^2)=3^2.(3 \times 3^2):8$. The nontrivial $3$-radical subgroups of
$J_3$ are the groups of type $\langle 3A \rangle$ and the Sylow $3$-subgroups of $J_3$; see \cite{kot}. Thus the
distinguished Bouc complex $|\widehat{\mathcal{B}}_3(J_3)|$ consists only of the $2^4 \cdot 5 \cdot 17 \cdot 19$
vertices corresponding to ${\rm Syl}_3(J_3)$. The fixed point set $|\widehat{\mathcal{B}}_3(J_3)|^{3A}$ can be computed by noting that $3A \in N_{J_3}(S)$ iff $3A \in S = C_{J_3}(3B^2)$ iff $3B^2 \leq C_{J_3}(3A) = 3 \times A_6$. There are precisely ten groups of type $3B^2$ in $3 \times A_6$, namely $Syl_3(A_6)$. The action of $PGL_2(9)$ on the ten vertices of $|\widehat{\mathcal{B}}_3(J_3)|^{3A}$ is the action on the projective line $\mathbb{F}_9 \cup \lbrace \infty \rbrace$.

Thus $\widetilde{L}_{PGL_2(9)}(|\widehat{\mathcal{B}}_3(J_3)|^{3A})$ is the projective irreducible Steinberg module and $\widetilde{L}_{N_{J_3}(\langle 3A \rangle)}(|\widehat{\mathcal{B}}_3(J_3)|^{3A})$ is indecomposable with vertex $\langle 3A \rangle$. The reduced Lefschetz module $\widetilde{L}_{J_3}(\widehat{\mathcal{B}}_3)$ contains one indecomposable summand with vertex $\langle 3A \rangle$, lying in a block with the same group as defect group.

The $3$-central elements in $J_3$ are closed since the multiplication class coefficient $\xi(3B, 3B, 3A) = 0$, but the Benson complex is slightly more complicated, containing $3B$ and $3B^2$. Since $|\widehat{\mathcal{B}}_3(J_3)|$ is homotopy equivalent to the Benson complex \cite[Theorem 3.1 and 4.4]{mgo3}, $|\widehat{\mathcal{B}}_3(J_3)|^{3B}$ is contractible. The reduced Lefschetz module $\widetilde{L}_{J_3}(\widehat{\mathcal{B}}_3)$ contains no summands with a vertex containing a $3$-central element. Since $C_{J_3}(3B) = 3 \times A_6$ and all $3$-elements in $A_6$ are $3$-central, the group $3= \langle 3A \rangle$ is the only purely noncentral subgroup of $J_3$. Thus $\widetilde{L}_{J_3}(\widehat{\mathcal{B}}_3)$ has only one nonprojective summand.
\end{proof}

\bibliographystyle{amsalpha}

\end{document}